\documentclass[runningheads, a4paper]{llncs}

\usepackage{amssymb}
\usepackage{graphicx}

\usepackage{url}
\newcommand{\keywords}[1]{\par\addvspace\baselineskip
\noindent\keywordname\enspace\ignorespaces#1}

\usepackage[sort&compress, numbers, square]{natbib}
\makeatletter
\renewcommand\bibsection%
{
  \section*{Bibliography}
}
\makeatother

\usepackage{amsmath, mathrsfs, multirow, subfigure}
\usepackage{graphicx, tikz} 
\usepackage{ifthen} 
\usepackage{amsfonts}
\usepackage{color}
\usepackage[normalem]{ulem}
\usepackage{tcolorbox}
\usepackage{soul}
\usepackage[ruled,vlined]{algorithm2e}

\newcommand{\prob}{\mathsf{P}} 
\newcommand{\E}{\mathsf{E}}

\newcommand{\unif}{{\sf Unif}}
\newcommand{\nm}{{\sf N}}

\newcommand{\RR}{\mathbb{R}}

\newcommand{\XX}{\mathbb{X}}

\newcommand{\TT}{\mathbb{T}}

\newcommand{\lPi}{\underline{\Pi}}

\begin{document}

\mainmatter 

\title{Which statistical hypotheses are afflicted with false confidence?}

\titlerunning{Which hypotheses are afflicted with false confidence?}

%

\author{Ryan Martin\thanks{Partially supported by the U.S.~National Science Foundation, SES--205122.}%
}
\authorrunning{R.~Martin}

\institute{
North Carolina State University \\ Department of Statistics \\
Raleigh, North Carolina 27695, U.S.A. \\ \email{rgmarti3@ncsu.edu}
}

%
%

\toctitle{Lecture Notes in Computer Science}
\tocauthor{Authors' Instructions}

\maketitle

\begin{abstract}
The false confidence theorem establishes that, for any data-driven, precise-probabilistic method for uncertainty quantification, there exists (non-trivial) false hypotheses to which the method tends to assign high confidence.  This raises concerns about the reliability of these widely-used methods, and shines new light on the consonant belief function-based methods that are provably immune to false confidence.  But an existence result alone is insufficient.  Towards a partial answer to the title question, I show that, roughly, complements of convex hypotheses are afflicted by false confidence. 

\keywords{Bayesian; consonant beliefs; convexity; inferential model; fiducial inference; possibility theory; validity.}
\end{abstract}

\section{Introduction}
\label{S:intro}

In {\em Logic of Statistical Inference}, Hacking \cite{hacking.logic.book} writes: ``Statisticians want numerical measures of the degree to which data support hypotheses.''  One such measure is a Bayesian posterior probability, but imprecise probabilists---especially those in the belief function community---are well aware that precise probability theory is not the only mode of uncertainty quantification.  Indeed, in a statistical inference problem, where prior information is at best incomplete and data speaks only indirectly through a model, there's good reason to question the appropriateness and/or reliability of a precise probability as statisticians' go-to quantitative expression of the degree to which data supports hypotheses.  

Balch et al.~\cite{balch.martin.ferson.2017} expressed this concern in terms of {\em false confidence}.  Roughly, in the context described in Section~\ref{SS:setup}, false confidence corresponds to the existence of false hypotheses to which, say, a default-prior Bayesian posterior distribution tends to assign high probability, support, or confidence.  Their result applies to (generalized) fiducial inference \cite{hannig.review, fisher1935a, dawid2020}, confidence distributions \cite{xie.singh.2012}, etc., so it highlights a risk of unreliability inherent in {\em all} precise-probabilistic approaches to statistical uncertainty quantification. Since reliability is obviously a top priority, there's an exciting opportunity for imprecise probability theory to make a fundamental contribution to statistics, a domain in which imprecise-probabilistic methods are greatly under-appreciated and largely unused.  Along these lines, I've recently shown \cite{martin.nonadditive, imchar, martin.partial, martin.partial2} that a suitable {\em possibilistic}, or {\em consonant belief} framework for statistical inference is immune to false confidence: it's reliable in the sense that it doesn't tend to assign high support to any false hypotheses!  

Unfortunately, the {\em false confidence theorem}, as stated in \cite{balch.martin.ferson.2017}, is only an existence result.  In a certain sense, the existence of hypotheses that are afflicted with false confidence is ``obvious,'' and it's partly for this reason that statisticians largely haven't taken this too seriously \cite{prsa.conf, prsa.response, williams.fct.2018}.  But the extent of false confidence affliction goes well beyond the hypotheses for which it's obvious: this has been demonstrated empirically in a number of specific examples, but no theoretical characterizations have been put forward.  In fact, to my knowledge, all that's known is that, for a relatively broad class of statistical models, certain hypotheses---basically, linear hypotheses---are safe from false confidence \cite{martin.isipta2023}.  The present paper partially answers the title question: {\em Which hypotheses are afflicted with false confidence?}  In particular, under a simple model that (approximately) represents most practical cases, I show that a class of (non-linear) hypotheses which includes those that are {\em co-convex}, i.e., complements of convex sets, are afflicted with false confidence.  This is not a complete characterization, but provides some insight as to what structure breeds false confidence. 

The present paper says very little about belief functions and imprecise probability, but I still expect this line of investigation to make a significant contribution.  Indeed, once the extent and implications of false confidence are understood, statisticians who care about reliable uncertainty quantification will have no choice but to use certain imprecise-probabilistic ideas and methods.

\section{Background}
\label{S:background}

\subsection{Problem setup}
\label{SS:setup}

Let $X$ denote the data taking values in a general sample space $\XX$.  A statistical model consists of a family of probability distributions $\{\prob_\theta: \theta \in \TT\}$ on $\XX$ indexed by a general parameter space $\TT$.  As is commonly assumed, suppose there is an uncertain ``true'' parameter value $\Theta$ such that $X$ has distribution $\prob_\Theta$.  I'll assume throughout that prior information about $\Theta$ is vacuous.  The high-level goal is to quantify uncertainty about $\Theta$, given $X=x$.

\subsection{Inferential models}

Following \cite{imchar}, an inferential model (IM) is a map from data $x \in \XX$ to a lower probability $\lPi_x$ supported on subsets of $\TT$, which depends implicitly on the statistical model and perhaps other things, e.g., prior information about $\Theta$.  The interpretation is that $\lPi_x(H)$ measures the degree of support for or belief/confidence in the truthfulness of the hypothesis ``$\Theta \in H$'' given data $X=x$.  Examples of precise IMs include Bayes's posterior probability and Fisher's fiducial distribution; examples of imprecise IMs include Dempster's seminal proposal \cite{dempster1966, dempster1968a, dempster2008}, Walley's generalized Bayes \cite{walley1991}, consonant likelihood-based belief functions \cite{shafer1982, wasserman1990b, denoeux2014}, and what's briefly described in Section~\ref{SS:poss.im} below.  

Bayesian- and fiducial-like frameworks quantify uncertainty about $\Theta$, given $X=x$, with a precise (countably additive) ``posterior distribution''
\begin{equation}
\label{eq:bayes}
\Pi_x(H) = \frac{\int_H L_x(\theta) \, \Pi(d\theta)}{\int_\TT L_x(\theta) \, \Pi(d\theta)}, \quad H \subseteq \TT, 
\end{equation}
where $\Pi$ is like a ``prior distribution'' for the uncertain parameter $\Theta$ and $\theta \mapsto L_x(\theta)$ is the model's likelihood function given $X=x$.  The quotation marks are highlight the point that, since prior information is assumed vacuous, these are ``prior'' and ``posterior'' distributions only in a formal sense.  As Jeffreys~\cite{jeffreys1946} explains, $\Pi$ is a default measure that gets updated to $\Pi_x$ by formally following Bayes's rule when $X=x$.  In certain contexts, e.g., \cite{hannig.review}, $\Pi$ itself might depend on data, hence can't represent genuine prior information.  In any case, the map $H \mapsto \Pi_x(H)$ is often used in applications to quantify uncertainty about $\Theta$, given $X=x$.  But ``[Bayes's rule] does not create real probabilities from hypothetical probabilities'' \cite{fraser.copss}, so a practically and theoretically important question is if this brand of probabilistic uncertainty quantification is reliable.

\subsection{False confidence}

The false confidence theorem \cite{balch.martin.ferson.2017, martin.nonadditive} says that, for {\em any} precise IM, $x \mapsto \Pi_x$, there exists a hypothesis--threshold pair $(H,\alpha)$ such that 
\[ \Theta \not\in H \quad \text{and} \quad \prob_\Theta\{ \Pi_X(H) \geq 1-\alpha \} > \alpha. \]
That is, there exists false hypotheses $H$ to which the posterior tends to assign a (relative to $\alpha$) large probability/confidence, which sheds light on a lurking unreliability.  Martin~\cite{imunder} shows that false confidence implies an incoherence-like risk of monetary loss to those who quantify uncertainty using precise IMs.  


\subsection{Consonant beliefs to the rescue}
\label{SS:poss.im}

Fisher \cite{fisher1930} writes: ``...[the likelihood function] does not obey the laws of probability; it involves no differential element.''  The default measure $\Pi$ also has no meaningful differential element ``$d\theta$''---with vacuous prior information, there's no reason to think that measure-theoretically larger hypotheses are ``more likely'' than smaller ones.  And if neither the likelihood nor the default prior have a meaningful differential element, then the differential element on the right-hand side of \eqref{eq:bayes} can't be meaningful either.  Indeed, it's easy to find very large hypotheses that are false, hence the trivial cases of false confidence.  More generally, the meaningless differential element is at the heart of false confidence \cite{martin.basu}. 

If there exists an IM that protects all hypotheses from false confidence, then it must be imprecise; the above remarks suggest that it should also be differential element free.  The simplest example is a consonant belief function \cite{shafer1976}, one whose conjugate plausibility function is a maxitive possibility measure \cite{hose2022thesis, dubois.prade.book}.  To my knowledge, the first IMs shown to satisfy 
\begin{equation}
\label{eq:valid}
\sup_{\Theta \not\in H} \prob_\Theta\{ \lPi_X(H) \geq 1-\alpha \} \leq \alpha, \quad \text{for all $(H,\alpha)$}, 
\end{equation}
were those put forward in \cite{imbook} based on nested random sets; see, also, \cite{denoeux.li.2018, balch2012}.  The condition \eqref{eq:valid} implies, among other things, that there's no false confidence.  The valid IM construction has been generalized and streamlined in \cite{plausfn, gim, martin.partial2}, but the specific details won't be needed in what follows.

\section{Co-convexity breeds false confidence}

Without much loss of generality, I'll focus here on the $D$-dimensional Gaussian case $X \sim \nm_D(\Theta, \Sigma)$, where $\Theta \in \TT = \RR^D$ is the uncertain parameter and the $D \times D$ covariance matrix $\Sigma$ is fixed and known.  Then the likelihood is 
\[ L_X(\theta) \propto \exp\{-\tfrac12 (X - \theta)^\top \Sigma^{-1} (X - \theta)\}, \quad \theta \in \TT. \]
I say ``without much loss of generality'' because, in most of the statistical models used in practical applications, there's a corresponding Gaussian limit experiment; see, e.g., Chapter~9 in \cite{vaart1998}. That is, if the sample size is large, then the maximum likelihood estimator (say) is an approximately minimal sufficient statistic whose sampling distribution is approximately Gaussian with mean $\Theta$ and covariance matrix a multiple of the inverse Fisher information.  In this case, with vacuous prior information, the go-to precise IM for $\Theta$ is 
\begin{equation}
\label{eq:posterior}
\Pi_X = \nm_D(X, \Sigma). 
\end{equation}
The precise IM in \eqref{eq:posterior} has a number of desirable properties, e.g, highest posterior density credible sets are minimum volume confidence sets.  But it still suffers from the inherent unreliability exposed by the false confidence theorem.

To develop some intuition, consider a function $\phi: \TT \to \RR$, and define
\begin{equation}
\label{eq:H.phi}
H_\phi = \{ \theta \in \TT: \phi(\theta) > \phi(\Theta)\}. 
\end{equation}
Clearly, hypothesis $H_\phi$ is {\em false}, i.e., $H_\phi \not\ni \Theta$.  If $\phi$ is (non-linear) convex, which makes $H_\phi$ {\em co-convex}---the complement of a convex set---then Jensen's inequality gives $\E_\Theta\{ \phi(X) \} > \phi(\Theta)$.  So, there must be non-negligible probability that $X$, the $\Pi_X$-posterior mean, is contained in the false $H_\phi$; and, if the posterior mean is in $H_\phi$, then the corresponding posterior probability, $\Pi_X(H_\phi)$, really can't be small, hence false confidence in $H_\phi$.  Interestingly, this apparently has little to do with the size of $H_\phi$ or $H_\phi^c$: something else is driving false confidence.  

The more general, albeit less intuitive, result is presented next.  Define a set $G \subset \TT$ to be non-linear locally convex at $\vartheta$, or $\vartheta$-{\em noloco}, if $G$ contains $\vartheta$ on its boundary, if it has a supporting hyperplane at $\vartheta$, and if the intersection of $G^c$ with the linear space that contains $G$ determined by the supporting hyperplane has non-zero Lebesgue measure; see Figure~\ref{fig:noloco}.  For example, if $\phi$ is a non-linear convex function, then the complement of $H_\phi$ in \eqref{eq:H.phi} is $\Theta$-noloco.  

\begin{figure}[t]
\begin{center}
\begin{tikzpicture}
\draw plot [smooth cycle] coordinates {(0,0) (1,1) (3,1) (1.5,0) (2,-1)};
\filldraw[black] (0.25,-0.23) circle (1pt) node[anchor=north]{$\vartheta$};
\draw[gray, very thin] (-1.2, 0.66) -- (2, -1.33);
\filldraw[black] (1,0) circle (0pt) node[anchor=south]{$G$};
\end{tikzpicture}
\end{center}
\caption{A non-convex $G$ that's $\vartheta$-noloco; gray line defines the supporting hyperplane.}
\label{fig:noloco}
\end{figure}
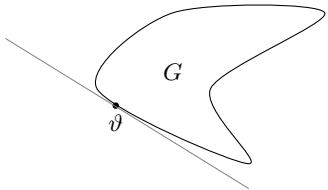

\begin{proposition}
\label{prop:fc}
For any $\Theta \in \TT$, if $G$ is $\Theta$-noloco, then the hypothesis $H=G^c$ is afflicted by false confidence.  In particular, the random variable $\Pi_X(H)$, as a function of $X \sim \nm_D(\Theta, \Sigma)$, is stochastically larger than $\unif(0,1)$. 
\end{proposition}

\begin{proof}
Let $g_\Theta$ denote the vector that defines the supporting hyperplane of $G$ at $\Theta$.  Since $G$ is contained in the half-space $\{\theta \in \TT: g_\Theta^\top(\theta-\Theta) \leq 0\}$, we get  
\[ H \supset H_\text{lin} := \{\theta \in \TT: g_\Theta^\top(\theta - \Theta) > 0\}, \]
and, consequently, $\Pi_X(H) > \Pi_X(H_\text{lin})$.  The last inequality is strict because $\Pi_X$ is absolutely continuous with respect to Lebesgue measure and, by assumption, $H \setminus H_\text{lin}$ has positive Lebesgue measure.  The lower bound, $\Pi_X(H_\text{lin})$, satisfies
\begin{equation}
\label{eq:post.lin}
\Pi_X(H_\text{lin}) = 1 - F\Bigl( -\frac{g_\Theta^\top (X - \Theta)}{\{g_\Theta^\top \Sigma \, g_\Theta\}^{1/2}} \Bigr), 
\end{equation}
where $F$ is the standard normal distribution function.  As a function of $X \sim \nm_D(\Theta, \Sigma)$, the right-hand side of \eqref{eq:post.lin} is $\unif(0,1)$.  Therefore, $\Pi_X(H)$ is (strictly) lower-bounded by a $\unif(0,1)$ random variable, completing the proof. 
\end{proof} 

By no means is this a complete characterization of false confidence.  For one thing, it's absolutely not necessary for $\Theta$ to sit on the boundary of the hypothesis---I imposed this constraint just to make the analysis tractable.  Similar results are expected for hypotheses that miss $\Theta$ but not by too much.  More generally, I don't believe that noloco is fundamental to false confidence.  My conjecture is that all {\em non-linear} hypotheses about $\Theta$ have at least a mild case of false confidence---the reason being that non-linear mapping can warp the parameter space in such a way that probability assignments get pushed in one direction or another systematically.  Precisely and rigorously diagnosing the existence and severity of affliction remains an open question.

\section{Illustrations}

\begin{example}
\label{ex:length}
{\em Non-linear hypotheses.} Inference on the squared length of a normal mean vector is a classically challenging statistical problem \cite{stein1959}.  In the present context, if $\phi(\theta) = \|\theta\|^2$, then $H_\phi$ as in \eqref{eq:H.phi} is a (false) hypothesis about the mean vector's squared length; it's also co-convex and, by Proposition~\ref{prop:fc}, is afflicted with false confidence.  To see the extent of affliction, the {\sc cdf} $\pi \mapsto \prob_\Theta\{ \Pi_X(H_\phi) \leq \pi \}$ is shown in Figure~\ref{fig:two.examples}(a), where the dimension is $D=2$ and $\Theta$ is length 1.  Note that, in this case, $\Pi_X(H_\phi)$ is always greater than 0.6, even though $H_\phi$ is false.  For comparison, Figure~\ref{fig:two.examples}(a) also displays the {\sc cdf} of the valid IM's lower probability $\lPi_X(H_\phi)$ and it's clear that there's no false confidence. 
\end{example}

\begin{figure}[t]
\begin{center}
\subfigure[Example~\ref{ex:length}]{\scalebox{0.45}{\includegraphics{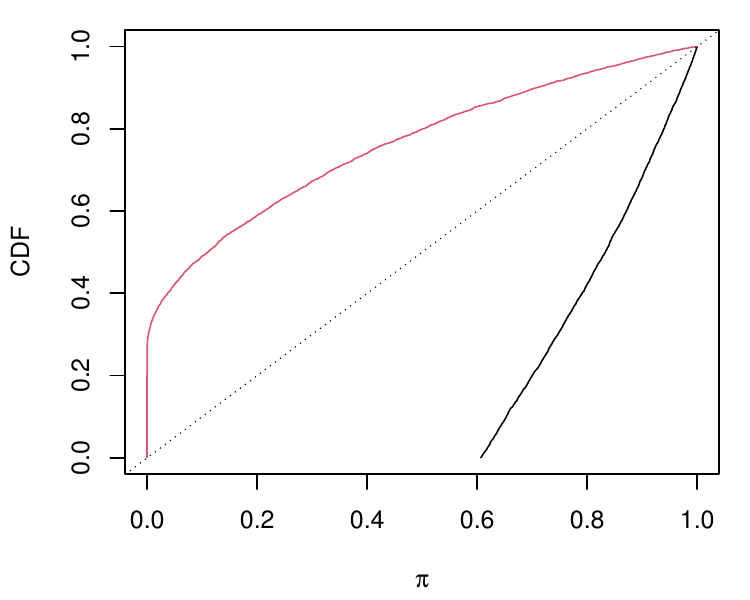}}}
\subfigure[Example~\ref{ex:bounded}]{\scalebox{0.45}{\includegraphics{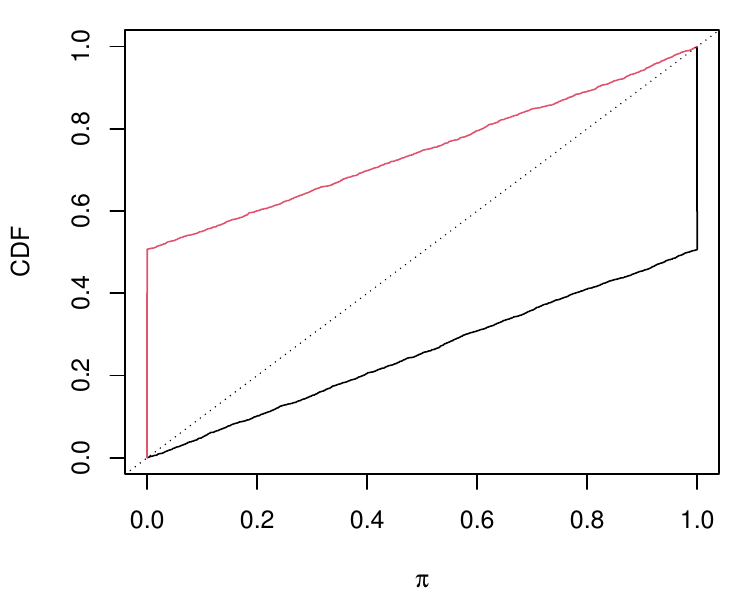}}}
\end{center}
\caption{Black lines are {\sc cdf}s for the (flat-prior) Bayes posterior probabilities and red lines are the {\sc cdf}s for the corresponding valid IM's lower probabilities.}
\label{fig:two.examples}
\end{figure}

\begin{example}
\label{ex:bounded}
{\em Non-linear parameter space.} Fraser~\cite{fraser2011} considers a normal mean model $X \sim \nm(\Theta, 1)$ but with the side information that $\Theta$ has a {\em known} lower bound, which I take to be 0 without loss of generality.  
Consider the (false) hypothesis $H = (\Theta, \infty)$---but note that the parameter constraint makes $H^c$ bounded.  The {\sc cdfs} of the (flat-prior) Bayesian posterior probability, $\Pi_X(H)$, and of the valid IM's lower probability, $\lPi_X(H)$, are shown in Figure~\ref{fig:two.examples}(b).  Note that the Bayes posterior assigns probability 1 to the false hypothesis $H$ 50\% of the time, while the valid IM does the polar opposite, rightfully assigning 0 (or small) support to the false hypothesis most of the time. 
\end{example}

\section{Conclusion}
\label{S:discuss}

The present paper is concerned with the following question: {\em which statistical hypotheses are afflicted by false confidence?}  My collaborators and I have had intuition about how to answer this question for some time, but only now have I been able to formulate this intuition in a way that's conducive to mathematical analysis.  The result that I proved here is quite simple, perhaps unremarkable, but I'd argue that simplicity is a virtue.  After all, false confidence is the rule, rather than the exception, so it should be easy to identify hypotheses that are afflicted.  What's interesting is that a property slightly more general than co-convexity is what makes the hypothesis vulnerable to false confidence.  

The result presented here provides a sufficient condition for false confidence, but I seriously doubt that the same condition is necessary.  As above, my conjecture is that non-linearity is enough to create at least a susceptibility to false confidence.  Non-linearlity alone may not be severe enough to cause false confidence-level problems as defined here; but maybe to cause the milder but still concerning ``fluke confidence'' that my friend and collaborator, Michael Balch, has been telling me about recently.  In any case, the advancements made in the present paper make me optimistic that we'll soon be able to settle these questions.

%
%
%

\bibliographystyle{apalike}

\bibliography{/Users/rgmarti3/Dropbox/Research/mybib}

\end{document}